\documentclass[12pt]{article}
\usepackage{amsmath, amsthm, amssymb}
\usepackage{fullpage}

\theoremstyle{plain}
\newtheorem{theorem}{Theorem}
\newtheorem{lemma}[theorem]{Lemma}

\newtheorem{conjecture}[theorem]{Conjecture}
\newtheorem{corollary}[theorem]{Corollary}

\usepackage{array,color,colortbl}

\def\cref#1{Conjecture~$\ref{#1}$}
\def\Cref#1{Corollary~$\ref{#1}$}

\renewcommand{\geq}{\geqslant}
\renewcommand{\leq}{\leqslant}

\def\dfrac#1#2{\lower0.15ex\hbox{\large$\frac{#1}{#2}$}} 

\title{Balanced diagonals in frequency squares}

\author{
Nicholas J. Cavenagh\thanks{
Department of Mathematics,
 The University of Waikato,
 Private Bag 3105,
 Hamilton 3240, New Zealand.}\\
\texttt{nickc@waikato.ac.nz}\\ 
\and
Adam Mammoliti
\thanks{
School of Mathematics and Statistics, UNSW Sydney, NSW 2052, Australia.}\\ 
\texttt{a.mammoliti@unsw.edu.au}\\ 
}

\begin{document}

\date{}
\maketitle

\begin{abstract}
We say that a diagonal in an array is {\em $\lambda$-balanced} if each entry occurs $\lambda$ times.
Let $L$ be a frequency square of type $F(n;\lambda^m)$; that is, an $n\times n$ array in which each entry from $\{1,2,\dots ,m\}$ occurs $\lambda$ times per row and $\lambda$ times per column.
We show that if $m\leq 3$, $L$ contains a $\lambda$-balanced diagonal, with only one exception  up to equivalence when $m=2$.  
We give partial results for $m\geq 4$ and suggest a generalization of Ryser's conjecture, that every latin square of odd order has a transversal. 
Our method relies on first identifying a small substructure with the frequency square that facilitates the task of locating a balanced diagonal in the entire array. 
\end{abstract}



\section{Introduction}
In what follows, rows and columns of an $n\times n$ array $L$ are each indexed by $N(n)=\{1,2,\dots ,n\}$, with 
$L_{i,j}$ denoting the {\em entry} in cell $(i,j)$. We sometimes consider an array~$L$ to be a set of ordered triples $L=\{(i,j;L_{i,j})\}$ so that the notion of a subset of an array is precise. 
A subarray of $L$ is any array induced by subsets of the rows and columns of $L$; thus the rows and columns in a subarray need not be adjacent. 

A {\em frequency square} or {\em $F$-square} $L$ of {\em type} $F(n;\lambda_1,\lambda_2,\dots ,\lambda_m)$ is an $n\times n$ array such that for each $i\in N(m)$, $i$ occurs $\lambda_i$ times in each row and $\lambda_i$ times in each column; necessarily $\sum_{i=1}^m \lambda_i=n$. 
In the case where $\lambda_1=\lambda_2=\dots =\lambda_m=\lambda$ we say that $L$ is of type $F(n;\lambda^m)$, where $n=\lambda m$.
Clearly a frequency square of type $F(n;1^n)$ is a Latin square of order $n$.

We define a {\em diagonal} in any square array to be a subset that uses each row and each column exactly once. 
We say that a diagonal is $\lambda$-balanced if each entry occurs  $\lambda$ times, for some $\lambda$. 
Thus, a $1$-balanced diagonal in a Latin square is precisely a transversal.

In this paper we restrict ourselves to frequency squares of type $F(n;\lambda^m)$; in this context we refer to a $\lambda$-balanced diagonal as simply being {\em balanced}; here each element of $N(m)$ appears exactly $\lambda$ times.  
For our purposes, two frequency squares of type $F(n;\lambda^m)$ 
 are equivalent if and only if one can be obtained from the other by rearranging rows or columns, relabelling symbols or taking the transpose.  

Trivially, any diagonal of a frequency square of type $F(\lambda;\lambda)$ is balanced. 
In Section 2 we show, with one exception up to equivalence, that each frequency square of type $F(2\lambda;\lambda,\lambda)$ has a balanced diagonal. In Section 3 we show that {\em every}  frequency square of type $F(3\lambda;\lambda,\lambda,\lambda)$ has a balanced diagonal.   
We then make some observations and conjectures about the existence of balanced diagonals in $F(m\lambda;\lambda^m)$ for $m >3$ in Section 4.

The existence of {\em transversals} in arrays (diagonals in which each entry appears at most once) or, equivalently, {\em rainbow matchings} in coloured bipartite graphs, has been well-studied \cite{ACH12,transurvey}. However there appears to be scarce results on the existence of diagonals in which each entry has a fixed number of multiple occurrences. 
Nevertheless, the existence of transversals (and other regular structures called {\em plexes}) in Latin squares imply the existence of balanced diagonals in certain frequency squares, as shown in Section 4. 
 Conversely, the results in this paper suggest a certain generalization of Ryser's conjecture, that each Latin square of odd order has a transversal; see Conjecture \ref{bigandbold}.  

 Research into frequency squares focuses mainly on constructing sets of pairwise orthogonal frequency squares, where each ordered pair occurs a constant number of times \cite{LM,LZD,Mav,JMM}. Here, in general, the existence of a balanced diagonal is not necessary.  ELABORATE? GIVE EXAMPLE?

Instead of starting with any diagonal and trying to permute rows and columns to make it balanced, we obtain our main results in Sections 2 and 3 by first identifying a subarray that allows us to construct a diagonal within the rest of the square that is {\em close} to being balanced. The properties of the subarray then allow us to find a balanced diagonal in the entire square. This approach makes the proof of Theorem \ref{main2} in particular delightfully terse (compared to an originally drafted much longer proof) and the proof of Theorem \ref{main3} manageable. This idea may be of use towards the solution of related combinatorial problems.  

\section{Balanced diagonals in frequency squares with $2$ symbols}

Let $A_{2\lambda}$ be the frequency square of type $F(2\lambda;\lambda,\lambda)$ 
with only $1$'s in the top-left and bottom-right quadrants, formally  
defined as follows: 
$$A_{2\lambda}=
\{(i,j;1),(i+\lambda,j+\lambda;1),(i,j+\lambda;2),(i+\lambda,j;2) \mid i,j\in N(\lambda)\}.$$ 

$$\begin{array}
{|c|c|c|c|c|c|}
\hline
1 & 1 & 1 & 2 & 2 & 2 \\
\hline
1 & 1 & 1 & 2 & 2 & 2 \\
\hline
1 & 1 & 1 & 2 & 2 & 2 \\
\hline
2 & 2 & 2 & 1 & 1 & 1 \\
\hline
2 & 2 & 2 & 1 & 1 & 1 \\
\hline
2 & 2 & 2 & 1 & 1 & 1 \\
\hline
\multicolumn{6}{c}{A_6} 
\end{array}$$
\begin{lemma}
The frequency array $A_{2\lambda}$ possesses a balanced diagonal if and only if $\lambda$ is even. 
\label{tobe}
\end{lemma}

\begin{proof}
It is easy to find a balanced diagonal if $\lambda$ is even. 
If $\lambda$ is odd, 
suppose that $A_{2\lambda}$ possesses a balanced diagonal $M$ with $x$ elements in cells $(i,j)$ where $i,j\in N(\lambda)$. 
  Then $M$ has $\lambda-x$ elements in cells $(i,j)$ where $i,j-\lambda\in N(\lambda)$ and in turn, 
$x$ elements in cells $(i,j)$ where  $i-\lambda,j-\lambda\in N(\lambda)$.
Thus, $2x$ elements of $M$ contain entry $1$, contradicting the fact that $\lambda$ is odd. 
\end{proof}

\begin{theorem}
Let $L$ be a frequency square of type $F(2\lambda;\lambda,\lambda)$. Then $L$ has a balanced diagonal, unless $L$ is equivalent to $A_{2\lambda}$ where $\lambda$ is odd. 
\label{main2}
\end{theorem}

\begin{proof}
Let $L$ be a frequency square of type $F(2\lambda;\lambda,\lambda)$. 
Observe that if $L$ does not possess the following subarray, it must be equivalent to $A_{2\lambda}$ and the previous lemma applies. 
$$\begin{array}{|c|c|}
\hline
1 & 1 \\
\hline
1 & 2 \\
\hline
\end{array}$$
Otherwise, we assume without loss of generality that 
$L_{1,1}=L_{1,2}=L_{2,1}=1$ and $L_{2,2}=2$.
 Let $M$ be the main diagonal and let $x$ be the number of $1$'s in $M$. Rearrange the rows and columns (except for the first two rows and columns) so that $|x-\lambda|$ is minimized. 
If $x-\lambda=0$ the main diagonal is balanced and we are done. 
If $x-\lambda=-1$, we can swap rows $1$ and $2$ and the main diagonal becomes a balanced diagonal. 
 
Suppose that $x-\lambda<-1$. Then there must exist $r,r'>2$ such that cells $(r,r)$ and $(r',r')$ each contain $2$ and $(r,r')$ contains $1$. Swapping rows $r$ and $r'$ reduces $|x-\lambda|$ by either $1$ or~$2$, a contradiction. If 
$x-\lambda\geq 1$, then there must exist $r,r'>2$ such that cells $(r,r)$ and $(r',r')$, each containing $1$, with $(r,r')$ containing $2$; if not, there is a $\lambda\times \lambda$ subarray containing only $1$'s and our array is equivalent to $A_{2\lambda}$, a contradiction. 
 Swapping rows $r$ and $r'$ reduces $|x-\lambda|$ by at least one unless $x-\lambda=1$ and $(r',r)$ contains $2$. In that case we can further swap rows $1$ and $2$ to create a balanced main diagonal.  
\end{proof}

\section{Balanced diagonals in frequency squares with $3$ symbols}

In this section we prove the following theorem. 
\begin{theorem}
Let $L$ be a frequency square of type $F(3\lambda;\lambda,\lambda,\lambda)$. Then $L$ has a balanced diagonal. 
\label{main3}
\end{theorem}
Throughout this section, $L$ is a frequency square of type $F(3\lambda;\lambda,\lambda,\lambda)$.

\begin{lemma}\label{lem: must contain subsquare}
Let $\lambda\geq 2$ and let $e\in \{1,2,3\}$. Suppose that for each $f\in \{1,2,3\}\setminus \{e\}$, 
$L$ does not contain the following subarray: 
$$
\begin{array}{|c|c|}
\hline
e & e \\
\hline
e & f \\
\hline
\end{array} \;.$$
Then $L$ contains a balanced diagonal. 
\label{nonny}
\end{lemma}

\begin{proof}
Without loss of generality, let $e=1$. It can quickly be shown that the rows and columns of $L$ can be partitioned into sets $R_1,R_2,R_3$ and $C_1,C_2,C_3$, respectively, where each set is of size $\lambda$ and the cells in $R_i\times C_i$ contain only symbol $1$, for each $1\leq i\leq 3$. 
Let $B_{ij}$ be the subarray formed by the intersection of the rows in $R_i$ and the columns in $C_j$; we call such subarrays {\em blocks}. 

We know that $B_{11}$, $B_{22}$, $B_{33}$ each contain only the symbol $1$. Suppose that one of the remaining blocks contains only the symbol $2$; it quickly follows that each block contains only one type of symbol. 
In this case diagonals from each of $B_{13}$, $B_{22}$ and $B_{31}$ together form a balanced diagonal. 
Otherwise, without loss of generality, each block not on the main diagonal contains at least one $2$ and at least one $3$. 

Let $X$ be the main diagonal from $B_{12}$  
(i.e. on the cells $\{(i,\lambda+i)\mid i\in N(\lambda)\}$), and let~$Y$ be the main diagonal from $B_{21}$.
Let $x$ and $y$ be the number of $3$'s in diagonals $X$ and $Y$, respectively.
If $x+y=\lambda$, we can construct a balanced diagonal by adding a diagonal from block $B_{33}$. 
Otherwise, without loss of generality, $x+y> \lambda$. 
We call a rearrangement of rows and columns that preserves the block structure  {\em good} if it reduces $x+y$ by at most $x+y-\lambda$. 
 
Partition $C_{2}$ into sets of columns $D_1$ and $D_2$ and $R_{1}$ into sets of 
rows $S_1$ and $S_2$, so that $S_{i-1}\times D_{i-1}$ contains the $i$'s from $X$, 
where $i\in \{2,3\}$. 
Similarly, 
partition $R_{2}$ into sets of rows $T_1$ and $T_2$ and $C_{1}$ into sets of 
 columns $E_1$ and $E_2$, so that $T_{i-1}\times E_{i-1}$ contains the $i$'s from $Y$, 
where $i\in \{2,3\}$. 

Suppose that $x+y\geq \lambda+2$ and there are no good swaps. 
Then there exists no $2$ in $S_2\times D_2$; otherwise we can swap rows within $S_2$ and columns within $D_2$ to reduce $x+y$ by either $1$ or~$2$. 
Similarly, $T_2\times E_2$ contains only $3$'s. 
If both $(r,c) \in S_1 \times D_2$ and $ (c-\lambda,r+\lambda)) \in S_2\times D_1$ contain a $2$, 
then swapping row $r$ with row $c-\lambda$ is a good swap; a contradiction. 
Thus, at most one of $(r,c) \in S_1 \times D_2$ and $(c-\lambda,r+\lambda) \in S_2\times D_1$ contains a $2$. 
Similarly, 
at most one of $(r,c) \in T_1 \times E_2$ and $ (c+\lambda,r-\lambda) \in T_2\times E_1$ contains a $2$.
Hence, the number of $3$'s in $B_{12}\cup B_{21}$ is at least $x^2+y^2+x(\lambda-x)+y(\lambda-y)=\lambda(x+y)>\lambda^2$. 

However, since each column contains $\lambda$ $3$'s, if there are $\alpha$ $3$'s in $B_{12}$, there are $\lambda^2-\alpha$ $3$'s in $B_{32}$. In turn, there are $\alpha$ 
$3$'s in $B_{31}$ and $\lambda^2-\alpha$ $3$'s in $B_{21}$. 
Therefore, the total number of~$3$'s in $B_{12}\cup B_{21}$ is equal to $\lambda^2$, a contradiction. 

So if $x+y>\lambda+1$ a good swap always exists; recursively we can always apply a 
series of good swaps until $x+y\in \{\lambda,\lambda+1\}$.  
So we are left with the case when $x+y=\lambda+1$. 
Since $\lambda\geq 2$, either $x\geq 2$ or $y\geq 2$. By symmetry we may assume that $x\geq 2$. Then, since $2\leq x\leq \lambda$, we have that $1\leq y\leq \lambda-1$. Thus, 
$X$ has distinct cells $(r_1,c_1),(r_2,c_2)$ each containing $3$ and 
$Y$ has a cell $(r_3,c_3)$ containing $2$ and a cell $(r_4,c_4)$ containing 
$3$. 

Since cell $(r_4,c_4)$ contains $3$, there exists a row $r_5\in R_3$ such that 
$(r_5,c_4)$ contains $2$. 
Suppose that there exists a column $c_5\in C_3$ such that $(r_2,c_5)$ contains $3$.
Let $Z$ be a diagonal from $B_{3,3}$ which includes $(r_5,c_5)$.  
We can construct a balanced diagonal  
by including $X$ (except for cell $(r_2,c_2)$),  $Y$ (except for cell $(r_4,c_4)$), $Z$ (except for cell $(r_5,c_5)$) and the 
cells $(r_4,c_2)$, $(r_2,c_5)$ and $(r_5,c_4)$.  
 
Otherwise, row $r_2$ contains only $2$'s within $B_{13}$. 
Similarly, column $c_4$ contains only $2$'s within $B_{31}$. 
Therefore cell $(r_3,c_4)$ contains $3$. 
Thus, there exists a column $c_6\in C_3$ such that $(r_3,c_6)$ contains $2$. 
Next, there exists a row $r_6\in R_3$ such that $(r_6,c_1)$ contains $2$.  
Let $Z$ be a diagonal from $B_{3,3}$ which includes $(r_6,c_6)$.  

Finally, we can construct a balanced diagonal 
by including $X$ (except for cell $(r_1,c_1)$),  $Y$ (except for cell $(r_3,c_3)$), 
$Z$ (except for cell $(r_6,c_6)$ and 
the  
cells $(r_6,c_1)$, $(r_3,c_6)$ and $(r_1,c_3)$.
\end{proof}

We next do the case $\lambda=2$. 

\begin{theorem}
Any frequency square of type $F(6;2,2,2)$ contains a balanced diagonal. 
\label{case6222}
\end{theorem}

\begin{proof}
Let $L$ be a frequency square of type $F(6;2,2,2)$. From Lemma \ref{nonny}, we may assume without loss of generality  
that $L_{1,1}=L_{1,2}=L_{2,1}=1$ and $L_{2,2}=2$. 
Since there are four~$3$'s in rows $1$ and $2$ and another four $3$'s in columns $1$ and $2$, there are four $3$'s in the block $\{3,4,5,6\}\times \{3,4,5,6\}$ and without loss of generality we can assume $L_{3,3}=L_{4,4}=3$. 

Suppose that there are no $3$'s in the block $\{5,6\}\times \{5,6\}$. 
If there exists at least one $2$ in this block, we can find a balanced diagonal. But if there are only $1$'s in this block the resultant partial structure does not complete to a frequency square of type $F(6;2,2,2)$. 

Thus, without loss of generality, $L_{5,5}=3$. If $\{L_{3,6},L_{6,3}\}=\{1,2\}$ or $\{2,2\}$ we can construct a balanced diagonal. 
Similarly, if $\{L_{4,6},L_{6,4}\}$ or $\{L_{5,6},L_{6,5}\}$ is equal to $\{1,2\}$ or $\{2,2\}$  we are done. Therefore, without loss of generality,
we may assume that $L_{3,6}=L_{4,6}=L_{6,3}=L_{6,4}=1$, $L_{6,5}=3$ and $L_{5,6}=2$. 
Hence, there are no more $3$'s in the block $\{3,4,5,6\}\times \{3,4,5,6\}$ and we are forced to have 
$L_{6,6}=2$, $L_{1,6}=L_{2,6}=3$ and $L_{1,5}=2$. 

If $L_{5,4}=2$ there is a balanced diagonal on cells $(1,1),(2,2),(3,3),(4,6),(5,4)$ and $(6,5)$. Similarly we are done if $L_{5,3}=2$. Otherwise $L_{5,4}=L_{5,3}=1$ and in turn, considering where a $1$ can be placed in row $2$, $L_{2,5}=1$.   
If $L_{5,2}=2$, there is a balanced diagonal on cells $(1,1),(2,5),(3,3),(4,4),(5,2)$ and $(6,6)$. Otherwise $L_{5,2}=3$. 
Thus, $L_{5,1}=2$; a balanced diagonal is shown underlined below. 
$$\begin{array}{|c|c|c|c|c|c|}
\hline
1 & \underline{1} & & & 2 & 3 \\
\hline
1 & 2 & & & \underline{1} & 3 \\
\hline
& & \underline{3} &  & & 1 \\
\hline
 & &  & \underline{3} & & 1 \\
\hline
\underline{2} & 3 & 1 & 1 & 3 & 2 \\
\hline
 &  & 1 & 1 & 3 & \underline{2} \\
\hline
\end{array}$$

\end{proof}

\begin{lemma}\label{lem: must contain two subsquares}
Let $\lambda\geq 3$. Suppose that $L$ contains the following subarray:
$$
\begin{array}{|c|c|}
\hline
e & e \\
\hline
e & f \\
\hline
\end{array} \;,
$$
for some $e,f\in\{1,2,3\}$ such that $e\neq f$.

Let $\{g\}=\{1,2,3\}\setminus \{e,f\}$. 
Then either $L$ contains one of the following subarrays (on rows and columns disjoint to the above) or $L$ contains a balanced diagonal.  
$$
\begin{array}{|c|c|}
\hline
g & g \\
\hline
g & e \\
\hline
\end{array}
\quad \quad
\begin{array}{|c|c|}
\hline
g & g \\
\hline
g & f \\
\hline
\end{array}\;.$$
\end{lemma}

\begin{proof}
Let $e=1$, $f=2$, $g=3$,  
$L_{1,1}=L_{1,2}=L_{2,1}=1$ and $L_{2,2}=2$. 
We assume that whenever $L_{r,c}=L_{r,c'}=L_{r',c}=3$, where $r,r',c,c'\geq 3$, then 
$L_{r',c'}=3$. 
If we can show that such an $L$ always contains a balanced diagonal, we are done. 

Since $\lambda\geq 3$, we can assume, without loss of generality, that  
the set of rows is $\{1,2\}\cup\bigcup_{i=1}^{\alpha} R_i$ and the set of columns is
 $\{1,2\}\cup \bigcup_{i=1}^{\alpha}C_i$ where 
the cells in $R_i\times C_i$ each contain $3$,  $1\leq i\leq \alpha$. 

Let $|R_i|=a_i$ and $|C_i|=b_i$ where $1\leq i\leq \alpha$. 
Then clearly $a_i,b_i\in\{\lambda,\lambda-1,\lambda-2\}$ for each $1\leq i\leq \alpha$ and 
\begin{eqnarray}
\sum_{i=1}^{\alpha} a_i & = & \sum_{i=1}^{\alpha} b_i=3\lambda-2.
\label{eqq1}
\end{eqnarray} 
(Note that since $\lambda>2$, this implies that $\alpha\geq 3$.)
For each $1\leq i\leq \alpha$, the number of $3$'s in $\{1,2\}\times C_i$ is equal to $(\lambda-a_i)b_i$; thus since there are no $3$'s in the subarray 
$\{1,2\}\times \{1,2\}$,
\begin{eqnarray}
0=2\lambda-\sum_{i=1}^{\alpha} (\lambda-a_i)b_i 
& = & \sum_{i=1}^{\alpha} a_i'b_i'-(\alpha-3)\lambda^2, 
\label{eqq2}
\end{eqnarray}
where $a_i'=\lambda-a_i\in \{0,1,2\}$ and $b_i'=\lambda-b_i\in \{0,1,2\}$, for each $i\in N(\alpha)$. 
Thus $0\leq 4\alpha-(\alpha-3)\lambda^2$. 

Suppose $\alpha>3$. Then $12/(\alpha-3)+4\geq \lambda^2$, so by inspection 
we must have $(\lambda,\alpha)\in \{(3,4),(3,5),(4,4)\}$. 
In the case $(\lambda,\alpha)=(4,4)$,  observe that the system of equations:
$$\sum_{i=1}^4 a_i'=\sum_{i=1}^4 b_i'=6,\quad \sum_{i=1}^4 a_i'b_i'=16$$
under the constraint $a_i' ,b_i' \in \{0,1,2\}$ for $1 \leq i \leq \alpha $ has no solution. 
By a similar observation, the case when $(\lambda,\alpha)=(3,5)$ 
is also impossible and, 
 without loss of generality,  the only possibility is when $(\lambda,\alpha)=(3,4)$ and  $a_1=a_2=b_1=b_2=1$, $a_3=b_3=2$ and $a_4=b_4=3$:
$$
\begin{array}{|c|c|c|c|c|c|c|c|c|}
\hline
1 & 1 & & & & & & &\\
\hline
1 & 2 & & & & & & &\\
\hline
 &  & 3 &  & & & & &\\
\hline
 &  &  & 3 & & & & &\\
\hline
 &  &  &  & 3& 3 & & &\\
\hline
 &  &  &  & 3& 3 & & &\\
\hline
 &  &  &  & & & 3& 3& 3\\
\hline
 &  &  &  & & & 3& 3& 3\\
\hline
 &  &  &  & & & 3& 3& 3\\
\hline
\end{array}\;,
$$
with no other $3$'s outside rows $1,2$ and columns $1,2$. 
 Then clearly the entries in $\{1,2\}\times\{3,4\}$ and $\{3,4\}\times\{1,2\}$ are each $3$.
Observe that column $9$ has no $3$'s in rows $1$ and $2$.  
If $L_{1,9}=L_{2,9}$, then we are done by considering cells $\{1,2\}\times \{2,9\}$ and cells $\{3,4\}\times \{1,4\}$. Otherwise $L_{1,9}\neq L_{2,9}$; again we are done by considering cells $\{1,2\}\times \{1,9\}$ and cells 
$\{3,4\}\times \{2,3\}$. 

Thus $\alpha=3$. Then, as above,
$$\sum_{i=1}^3 a_i'=\sum_{i=1}^3 b_i'=2, \quad \sum_{i=1}^3 a_i'b_i'=0.$$ 
If 
$\{a_i',a_2',a_3'\}=\{b_1',b_2',b_3'\}=\{0,0,2\}$, then 
$\{a_1,a_2,a_3\}=\{b_1,b_2,b_3\}=\{\lambda,\lambda,\lambda-2\}$ and without loss of generality $a_1=a_2=b_1=b_3=\lambda$ and $a_3=b_2=\lambda-2$.  
It follows that whenever~$3$ is in cells $(r,c)$, $(r,c')$ and $(r',c)$, it is also in cell $(r',c')$. So a balanced diagonal exists  
  by Lemma \ref{nonny}. 
Thus, without loss of generality, we may otherwise assume that 
$a_1=\lambda-2$, $b_2=b_3=\lambda-1$ and $a_2=a_3=b_1=\lambda$. 
Clearly $\{1,2\}\times C_1$ contains only $3$. It follows there exists a column $c'\in C_2\cup C_3$ such that 
$(2,c')$ contains $1$ and $(1,c')$ contains $2$. Thus we have a subarray of the form 
$$\begin{array}{|c|c|}
\hline
1 & 2 \\
\hline
1 & 1 \\
\hline
\end{array}$$
within rows $1$ and $2$, using column $1$ and one column from $C_2\cup C_3$. 

Next, consider $R_2\times \{1,2\}$. If all the $3$'s are in one column in this subarray, then, similarly to before, we are done by Lemma \ref{nonny}. 
Otherwise, there exists 
 a subarray of the form 
$$\begin{array}{|c|c|}
\hline
e & 3 \\
\hline
3 & 3 \\
\hline
\end{array}$$
where $e\in \{1,2\}$, the rows are in $R_2$, column $2$ and one column from $C_2$ can be used. Since $\lambda\geq 3$, this can be made to avoid 
the previous subarray, and 
we are done. 
\end{proof}

\begin{lemma}
Let $\lambda\geq 3$. If $L$ does not contain a balanced diagonal, then $L$ contains a $4\times 4$ subarray with (possibly non-disjoint) diagonals $D_1,D_2,D_3$ and $D_4$, where:
\begin{itemize}
\item $D_1$ contains symbol $1$ twice and each other symbol once;
\item $D_2$ contains symbol $2$ twice and each other symbol once;
\item $D_3$ contains symbol $3$ twice and each other symbol once; and 
\item $D_4$ contains two symbols twice each, with the remaining symbol not included. 
\end{itemize}
\label{setup}
\end{lemma}

\begin{proof}
If $L$ contains row- and column- disjoint subarrays equivalent to:
$$
\begin{array}{|c|c|}
\hline
1 & 1 \\
\hline
1 & 2 \\
\hline
\end{array}
\quad \quad
\begin{array}{|c|c|}
\hline
3 & 3 \\
\hline
3 & 2 \\
\hline
\end{array}$$
observe that we are done. 
Moreover, if 
 $L$ contains row and column disjoint subarrays equivalent to:
$$
\begin{array}{|c|c|}
\hline
1 & 2 \\
\hline
1 & 3 \\
\hline
\end{array}
\quad \quad
\begin{array}{|c|c|}
\hline
2 & 3 \\
\hline
2 & 1 \\
\hline
\end{array}$$
observe that we are done. We call these Cases 1 and 2, respectively and refer to them later in the proof. 
In Case 2, the transpose of either $2\times 2$ subarray is also allowed. 

From Lemmas \ref{lem: must contain subsquare} and \ref{lem: must contain two subsquares}, 
we are left, without loss of generality, with the case when 
there is a substructure as follows:
$$
\begin{array}{|c|c|c|c|}
\hline
1 & 1 & & \\
\hline
1 & 2 & & \\
\hline
 &  & 2 & 2 \\
\hline
 &  & 2 & 3 \\
\hline
\end{array}\;, 
$$
on rows $1,2,3$ and $4$ and columns $1,2,3$ and $4$. 
Note that if a diagonal of the form $D_3$ exists in this subarray we are done.  

Let the remaining sets of rows and columns be $R'$ and $C'$ respectively. 
If there is a column $c\in C'$ such that $(3,c)$ and $(4,c)$ each contain $3$, we 
obtain Case 1 by swapping columns $3$ and $c$ and we are done. 
Similarly, if there is a row $r\in R'$ such that $(3,r)$ and $(4,r)$ each contain $3$, we are done. 

{\bf Claim:}\ {\em Either} there exists a column 
 $c'\in C'$ such that the set of entries in $(3,c')$ and $(4,c')$ is $\{1,3\}$; {\em or} 
there exists a row $r'\in R'$ such that the set of entries in $(r',3)$ and $(r',4)$ is $\{1,3\}$. 
If not, for each occurrence of entry $3$ in a cell $(3,c')$ where $c'\in C'$, entry $2$ must occur in $(4,c')$. 
Also, 
 for each occurrence of entry $3$ in a cell $(4,c')$ where $c'\in C'$, entry $2$ must occur in $(3,c')$. 
Since cells $(3,3)$, $(3,4)$ and $(4,3)$ each contain $2$, 
there are at most $\lambda-1$ columns in $C'$ containing $3$ in row $3$ and 
at most $\lambda-2$ columns in $C'$ containing $3$ in row $4$. 
Thus, there exists at least one 3 in cells $(3,1)$ or $(3,2)$ and at least one 3 in cells 
$(4,1)$ or $(4,2)$. 
By symmetry, there 
 exists at least one 3 in cells $(1,3)$ or $(2,3)$ and at least one three in cells 
$(1,4)$ or $(2,4)$. 

If $3$ is in cell $(4,2)$, then $D_3$ exists either on the set of cells $\{(1,1),(2,3),(3,4),(4,2)\}$ or on the set of cells $\{(1,3),(2,1),(3,4),(4,2)\}$. Thus, $3$ is not in cell 
$(4,2)$ and $3$ is in cell $(4,1)$. Similarly, $3$ is in cell $(1,4)$ (and not in cell $(2,4)$). If $3$ is in cell $(3,2)$ then $D_3$ exists on cells 
$\{(1,4),(2,1),(3,2),(4,3)\}$. Thus, $3$ is in cell $(3,1)$ (and not in cell $(3,2)$). 

By further inspection, the only possible scenario that does not allow diagonal $D_3$ to exist in 
the first four rows and columns is:
$$
\begin{array}{|c|c|c|c|}
\hline
1 & 1 & 3 & 3 \\
\hline
1 & 2 & 1 & 1 \\
\hline
3  & 1 & 2 & 2 \\
\hline
3 & 1 & 2 & 3 \\
\hline
\end{array} \;. 
$$
But then we are done (as in Case $1$), considering the entries in cells 
$\{1,2\}\times \{1,4\}$ and 
$\{3,4\}\times \{3,c'\}$, since there exists $c'$ such that $(3,c')$ and $(4,c')$ contain entries $3$ and $2$, respectively.  

Hence, our claim is true. Without loss of generality, 
 there exists a column 
 $c'\in C'$ such that 
the set of entries in $(3,c')$ and $(4,c')$ is $\{1,3\}$.
If there exists a column $c''\in C'$ such that 
$c''\neq c'$ and 
cells $(1,c'')$ and $(2,c'')$ each contain $3$ we are done,  
 since we have Case 2 on 
cells $\{1,2\}\times \{2,c''\}$ and 
cells $\{3,4\}\times \{3,c'\}$. 
If there exists a column $c''\in C'$ such that 
$c''\neq c'$ and 
the set of entries in $(1,c'')$ and $(2,c'')$ is $\{2,3\}$ we are also done,
 since we have Case 2 on 
cells
$\{1,2\}\times \{1,c''\}$ 
and 
cells $\{3,4\}\times \{3,c'\}$.

Since cells $(1,1)$ and $(1,2)$ each contain $1$, 
there are at most $\lambda-2$ columns in $C'\setminus \{c'\}$ containing $3$ in row $2$.   
It follows that $3$ either exists in cell $(2,3)$ or cell $(2,4)$. 

If there is a $3$ in cell $(4,2)$, then we can find diagonal $D_3$ and we are done. 
Therefore, $3$ occurs as an entry at least $\lambda-1> 1$ times in column $c_2$ (within rows $R'$). 

Let $r'\in R'$ be such that $(r',2)$ contains $3$. Suppose that 
$(r',1)$ contains $1$. Then we are back to Case 1 on cells 
$\{1,r'\}\times \{1,2\}$ and 
$\{3,4\}\times \{3,4\}$.  
If $(r',1)$ contains $2$, we are back to Case 2
on cells 
$\{1,r'\}\times \{1,2\}$ and 
$\{3,4\}\times \{3,c'\}$.  
Thus, $(r',1)$ contains $3$. 
But then we have Case 2 on cells 
$\{2,r'\}\times \{1,2\}$ and 
$\{3,4\}\times \{3,c'\}$. 
\end{proof}

In our proof of Theorem \ref{main3}, we also need the following lemma. 
\begin{lemma}
Let $L$ be a frequency square of type $F(3\lambda;\lambda,\lambda,\lambda)$ 
and suppose that there is a $k\times k$ subarray missing a particular entry. Then $k\leq 3\lambda/2$.  
\label{halfy}
\end{lemma}

\begin{proof}
Let $R$ and $C$ be subsets of the rows and columns, respectively, such that $|R|=|C|=k$ and 
the cells $R\times C$ are missing entry $1$ (say). Let $R'=N(3\lambda)\setminus R$ and 
$C'=N(3\lambda)\setminus C$. 
Then there are $k\lambda$ $1$'s in $R'\times C$ and in turn, $(3\lambda-2k)\lambda$ $1$'s in $R'\times C'$. Thus $2k\leq 3\lambda$.  
\end{proof}

Finally, we are ready to prove Theorem \ref{main3}:

\begin{proof}
The case $\lambda=1$ is trivial and the case $\lambda=2$ is done by Theorem \ref{case6222}.  
Otherwise suppose, for the sake of contradiction, that there exists a frequency square $L$ of type $F(3\lambda;\lambda,\lambda,\lambda)$ with no balanced diagonal, where $\lambda\geq 3$. 
From Lemma \ref{setup}, there is a $4\times 4$ subarray in~$L$ as described in the statement of that lemma. 
Let this subarray be on the set of rows $R_0=\{1,2,3,4\}$ and the set of columns $C_0=\{1,2,3,4\}$, with $R'$ and $C'$ the remaining sets of rows and columns. Observe that each entry occurs in at least two cells of $R_0\times C_0$. 

At any stage in the following, $M$ is the main diagonal on the subarray $R'\times C'$ and
$\alpha$, $\beta$ and $\gamma$ are the number of $1$'s, $2$'s and $3$'s, respectively, within $M$. 
Thus $\alpha+\beta+\gamma = 3\lambda-4$. Let $\Delta=\gamma-\alpha$.    
Let $R_1$ and $C_1$ be the sets of rows and columns, respectively, which contain $1$ on the main diagonal (within $M$), with $R_2,C_2,R_3$ and $C_3$ similarly defined. 
For $0\leq i,j\leq 3$ let $B_{i,j}$ be the block of cells $R_i\times C_j$.   
At certain steps we will perform a permutation of rows in $R'$ or columns of $C'$, changing the entries in $M$; 
we recalculate $\alpha$, $\beta$, $\gamma$, $\Delta$ (and in turn $R_1,R_2,R_3,C_1,C_2$ and $C_3$) accordingly, assuming always (by relabelling if necessary) that 
$\alpha\leq \beta\leq \gamma$. 

If $\Delta =1$ then $\alpha=\lambda-2$, $\beta=\gamma=\lambda-1$ and we are done by Lemma \ref{setup}. 
Otherwise $\Delta\geq 2$.  
We call a permutation of $R'$ or $C'$ {\em good} if it decreases the value of $\Delta$
{\em or} creates a balanced diagonal on $M$ together with $D_4$ (in the case $\Delta=2$). 
We are done if we can show that there always exists a good permutation, so we assume no such permutation exists for the sake of contradiction.  

\vspace{5mm}

\noindent {\bf Case 1}: $\alpha+3\leq \beta$. \
If there is a $1$ in 
cell $(r,c)$ in block $B_{2,3}$, then the permutation $(rc)$ (applied to the rows) is good. 
Similarly, there is no $1$ in block $B_{3,2}$. 
If there is a $1$ in 
cell $(r,c)$ in block $B_{3,3}$, then the permutation $(rc)$ (applied to the rows) is good, unless $\beta=\gamma$ and cell $(c,r)$ contains $2$.
In this latter case, 
if there is a $1$ in cell $(r',c')$ of block $B_{2,2}$, then 
the permutation $(r'c')(rc)$ applied to the rows is always good.
If there are no $1$'s in block $B_{2,2}$, 
then we globally swap entry $2$ with entry $3$ within $L$, noting that $\beta=\gamma$ still holds. 
In any case we may assume that $1$ is neither in block  
$B_{2,3}$ nor block $B_{3,3}$. 
 
Thus, the $1$'s in the columns of $C_3$ 
are in rows $R_0\cup R_{1}$. 
Therefore, $\alpha\geq \lambda-4$. 
If $\alpha=\lambda-4$,  every entry in block $B_{0,3}$ must be $1$; hence $\gamma\leq \lambda$.
If $\gamma <\lambda$, $\alpha+\beta+\gamma<3\lambda-4$, a contradiction. Hence, $\gamma=\lambda$. However, then there are no $1$'s in $B_{0,0}$, a contradiction.  Thus, $\alpha\geq \lambda-3$. But then $\beta,\gamma \geq \lambda$, which gives $\alpha+\beta+\gamma>3\lambda-4$, again a contradiction. 

\vspace{5mm}

\noindent {\bf Case 2}: $\alpha+2=\beta$. \
Note that $\gamma>\beta$; otherwise $\alpha+\beta+\gamma\not\equiv 2$ (mod $3$). 
Therefore, if there is a $1$ in block $B_{3,3}$ or $B_{2,3}$, a good permutation exists, similarly to the previous case.  
Furthermore, the 1's in the columns $C_3$ are in rows $R_0\cup R_1$.
Thus, $\alpha\geq \lambda-4$. If $\alpha=\lambda-4$, 
then the first row must contain $\gamma=\lambda+2>\lambda$ $1$'s. If $\alpha\geq \lambda-2$, 
then $\alpha+\beta+\gamma > 3\lambda-4$. Hence $\alpha=\lambda-3$, $\beta=\lambda-1$ and $\gamma=\lambda$. 

For each $1$ in a cell $(r,c)$ of block $B_{1,3}$, if there is also a $1$ in cell $(c,r)$ (of block $B_{3,1}$), the permutation $(rc)$ applied to the rows is good. A similar result holds for each $1$ in block $B_{3,1}$. Therefore, without loss of generality, we may assume there are at most $\lambda(\lambda-3)/2$ $1$'s in block $B_{1,3}$.  
 But the number of $1$'s in $B_{0,3}$ is at most $4\lambda-2$  
(since $B_{0,0}$ has at least two $1$'s). But the total number of $1$'s in columns $C_3$ is $\lambda^2$, so $(4\lambda-2)+\lambda(\lambda-3)/2\geq \lambda^2$. It follows that $\lambda<5$. 

Suppose that $\lambda=4$. Without loss of generality, suppose there are at most two $1$'s in block $B_{1,3}$. Thus, there are at least $14$ $1$'s in $B_{0,3}$. But there are already at least two $1$'s in $R_0\times C_0$. We thus have that 
$B_{0,2}$ has no $1$'s. But $B_{3,2}$ also has no $1$'s, 
 so each cell in $B_{1,2} \cup B_{2,2}$ contains 1, which is impossible
 as the main diagonal contains 2 within $B_{2,2}$. 

Otherwise $\lambda=3$. Then $R_1$ and $C_1$ are empty and $B_{0,3}$ contains nine $1$'s. Since $R_0\times C_0$ contains at least two $1$'s,  
$B_{0,2}$ contains at most one $1$. But there are no $1$'s in $B_{3,2}$ and at most two $1$'s in $B_{2,2}$ (as this block has $4$ cells), 
contradicting the fact that each column of~$C_2$ must have three $1$'s. 

\vspace{5mm}

\noindent {\bf Case 3}: $\alpha=\beta$ or $\alpha+1=\beta$ and $\alpha\leq \lambda-3$. \ 
Since $\alpha+\beta+\gamma=3\lambda-4$,  observe that $\gamma\geq \beta+3$. 
Since no good permutations exist, there are only $3$'s in the block $B_{3,3}$. But then $\gamma\leq \lambda$ and $\alpha+\beta+\gamma\leq 3\lambda-5$, a contradiction. 

\vspace{5mm}
\noindent {\bf Case 4}: $\alpha=\beta=\lambda-2$ and $\gamma=\lambda$. 
Recall the existence of the diagonal $D_4$ from the statement of Lemma \ref{setup}. If $D_4$ contains no $3$'s there exists a balanced diagonal. Hence, $D_4$ contains entry $3$ twice. 

First suppose there are only $3$'s in the block $B_{3,3}$. Then, to avoid a good permutation, there are only $1$'s in 
blocks $B_{2,3}$ and $B_{3,2}$ and only $2$'s in blocks $B_{1,3}$ and $B_{3,1}$. 
Since $\lambda\geq 3$, $R_1$, $R_2$, $C_1$ and $C_2$ are each non-empty and $|R_3|=|C_3|\geq 3$. 
Let $r_1\in R_1$, $r_2\in R_2$ and $r_3,r_4$ distinct rows from $R_3$.  
Then the permutation $(r_1r_3)(r_2r_4)$ (applied to the rows) is good, a contradiction.  

Thus, without loss of generality, there is a cell $(r,c)$ in $B_{3,3}$ containing $2$; the cell $(c,r)$ must also contain $2$ otherwise $(rc)$ is a good permutation. If~$D_4$ is missing $2$ we are done; therefore~$D_4$ has entry  $2$ twice and entry $3$ twice. 
If there is also a cell $(r,c)$ in $B_{3,3}$ containing~$1$, then again cell $(c,r)$ must contain $1$. However, we are now done by the existence of the diagonal~$D_4$. Therefore, $B_{3,3}$ does not contain entry $1$. 

Next, suppose there is a row $r_1$ of $B_{3,3}$ with no $2$'s. Then column $r_1$ contains no $2$'s in~$B_{3,3}$. There exists a cell $(r_{2},c_{3})$ of $B_{3,3}$ containing $2$. 
Then the permutation $(r_1r_{2}c_{3})$ is good when applied to the columns. Hence, {\em every} row and column of $B_{3,3}$ contains a $2$. 

Next, suppose that there is a $1$ in cell $(r_1,c_2)$ of $B_{2,2}$. 
Let $(r_3,c_4)$ be a cell of $B_{3,3}$ containing~$2$; from above, $(c_4,r_3)$ also contains $2$.  
Then the permutation $(r_1c_2)(r_3c_4)$ is always good, using $D_4$ in the case when $(c_2,r_1)$ contains $1$.  
Thus, there are no $1$'s in $B_{2,2}$.

Suppose that there is a $1$ in $B_{2,3}$ in cell $(r_{1},c_{2})$, say. 
Then there exists $r_{3}\in R_3$ such that $(r_{3},c_{2})$ contains $2$. 
Observe that cells $(r_{1},r_{1})$ and $(c_{2},r_{3})$ each contain $2$ and 
cells $(c_{2},c_{2})$ and $(r_{3},r_{3})$ each contain $3$. 
However regardless of whether cell $(r_{3},r_{1})$ contains a $1$, $2$ or a $3$, the permutation $(r_{1}c_{2}r_{3})$ (applied to the columns) is always good (possibly using $D_4$). 

Thus, there are no $1$'s in $B_{2,3}$; similarly there are no $1$'s in $B_{3,2}$. 
Therefore, by Lemma \ref{halfy}, $\lambda+(\lambda-2)\leq 3\lambda/2$ and $\lambda\leq 4$. 
If $\lambda=4$ then there are $24$ $1$'s in $B_{1,2}\cup B_{1,3}\cup B_{0,2}\cup B_{0,3}$ and thus there are no $1$'s in $B_{0,0}$, a contradiction. 

So we are left with the case $\lambda=3$. 
 Observe there are at least two $1$'s in $B_{0,2}$ (since $B_{2,2}\cup B_{3,2}$ contains no $1$'s and 
$|R_1|=1=|C_2|$). Hence, there are at most eight $1$'s in $B_{0,3}$ and at least one $1$ in $B_{3,1}$.
Similarly, there is at least one $1$ in $B_{1,3}$. 
If there is a $1$ in cell $(r,c)$ of $B_{1,3}$ and in cell $(c,r)$ of $B_{3,1}$, the permutation $(rc)$ applied to the rows is good. 
Thus, there exist rows $r_1\in R_1$ and $r_2,r_3\in R_3$ such that $(r_1,r_2)$ and $(r_3,r_1)$ each contain~$1$ and $r_2\neq r_3$. Observe that $(r_1,r_1)$ contains $1$, $(r_2,r_2)$ and $(r_3,r_3)$ each contain~$3$ and $(r_2,r_3)$ contains either~$2$ or $3$. Therefore, the permutation $(r_1r_2r_3)$ (applied to the columns) is good. 
\end{proof}

\section{Frequency squares with $m\geq 4$}

We conjecture the following. 

\begin{conjecture}
Let $L$ be a frequency square of type $F(m\lambda;\lambda^m)$. 
If $(m-1)\lambda$ is even,  
$L$ contains a balanced diagonal. 
\label{bigandbold}
\end{conjecture}

We have shown in the previous sections this conjecture to be true for $m\in {2,3}$ (it is trivially true for $m=1$).  
In this section we make some observations and constructions that support (or at least don't contradict) this conjecture and show its connection with existing conjectures and known results. 

The following conjecture is commonly known (including in this paper) as Ryser's conjecture, though, as pointed out in \cite{transurvey}, Ryser's original conjecture was that each Latin square of order $n$ has an odd number of transversals \cite{Ryser}. It is known to be true for $n\leq 9$ \cite{MMW06}. 

\begin{conjecture} {\rm (Ryser's conjecture)}\
Each Latin square of odd order $n$ has at least one transversal. 
\end{conjecture}
\noindent
Setting $\lambda=1$ with $m$ odd, observe that Conjecture \ref{bigandbold} implies Ryser's conjecture. 

Next, we explore a method to construct a frequency square from a Latin square, and explore how the existence or non-existence of transversals in the Latin square affects the existence or non-existence of balanced diagonals in the resultant frequency square. 
To this end,  given a Latin square $L$ of order $m$, let $L(\lambda)$ be the frequency square of type $F(m\lambda; \lambda^m)$ created by replacing each cell in $L$ containing entry $e$ with a $\lambda\times \lambda$ block of $e$'s. 

A {\em $k$-plex} in a Latin square of order $n$ is a subset in which each row, column and entry occurs exactly $k$ times. 
 Note that a $1$-plex is a transversal.

\begin{theorem}
If a Latin square has a $k$-plex and $k$ divides $\lambda$ then $L(\lambda)$ has a balanced diagonal. 
In particular, if a Latin square $L$ has a transversal, then $L(\lambda)$ has a balanced diagonal for any $\lambda$. 
\label{goodwork}
\end{theorem}

\begin{proof}
Let $L$ be a Latin square with a $k$-plex $K$ and let $\lambda/k=\alpha$. 
It is well-known that any regular bipartite graph partitions into perfect matchings; thus 
we may partition $K$ into diagonals  $D_1,D_2,\dots ,D_{k}$, noting, in general that 
these diagonals are not necessarily balanced (i.e. there may be repeated entries within a particular diagonal). 

For each cell $(r,c)$ in diagonal $D_i$, $1\leq i\leq k$, 
choose the entries in cells 
$$\{(\lambda(r-1)+j,\lambda(c-1)+j)\mid  i(\alpha-1) +1\leq j\leq i\alpha \}$$  
from $L(\lambda)$. Overall we have constructed a balanced diagonal in $L(\lambda)$. 
\end{proof}

Since any Latin square of order $n$ is by definition an $n$-plex, we have the following corollary. 

\begin{corollary}
Let $L$ be an $n\times n$ Latin square. If $n$ divides $\lambda$, the frequency square $L(\lambda)$ contains a balanced diagonal.  
\end{corollary}

The following generalization of Ryser's conjecture is implied by Conjecture 8.5 from \cite{transurvey}, originally given in \cite{Dough} (and attributed to Rodney \cite{CD}, p. 143); independently given also in~\cite{Wan02}. 

\begin{conjecture}
If $n$ is even, every Latin square of order $n$ has a $2k$-plex, for each $k\in N(n/2)$.
If $n$ is odd, every Latin square of order $n$ has a $k$-plex for each $k\in N(n)$. 
\end{conjecture}

If true this conjecture, together with Theorem \ref{goodwork},  would imply that if $L$ is a Latin square of order $n$, 
$L(\lambda)$ has a balanced diagonal whenever $(n-1)\lambda$ is even, supporting Conjecture~\ref{bigandbold} above. 

Conversely, the existence of Latin squares without any odd plexes is known, for example certain Cayley tables \cite{Wan02}:
\begin{theorem}
If $G$ is a group of finite order $n$ with a non-trivial cyclic Sylow $2$-subgroup, then the Cayley table of $G$ contains no $k$-plex for any odd $k$.   
\end{theorem}
For other results on the existence of Latin squares that do not contain plexes of certain sizes, see \cite{EW08,transurvey}.

Next, we construct frequency squares that do {\em not} possess a balanced diagonal. 
Let $B_n$ be the Latin square of order $n$ where cell $(i,j)$ contains $i+j$ (mod $n$); we further replace each $0$ with $n$ to be consistent with definitions earlier in the paper, for our purposes this is a cosmetic change. 
That is, $B_n$ is the addition table for the integers  modulo $n$. It is well-known that $B_n$ possesses a transversal if and only if $n$ is odd. 
Note that $B_2(\lambda)$ is equivalent to $A_{2\lambda}$ from Section 2. 
The following theorem is similar to the Delta lemma, first introduced in \cite{EW08} and \cite{Ev}. 

\begin{theorem}
The frequency square $B_n(\lambda)$ possesses a balanced diagonal if and only if $(n-1)\lambda$ is even. 
\end{theorem}

\begin{proof}
If $n$ is odd, the main diagonal of $B_n$ is a transversal. 
By Theorem \ref{goodwork}, $B_n(\lambda)$ has a balanced diagonal for any $\lambda$. 
If $n$ is even, a $2$-plex of $B_n$ is given by the main diagonal and the diagonal of cells of the form $(r,r+1)$, $r\in N(n)$ (where $n+1$ is replaced by $1$);   then apply Theorem \ref{goodwork}.  
If $\lambda$ is even, again apply Theorem \ref{goodwork} to obtain a balanced diagonal in $B_n(\lambda)$. 

Otherwise suppose that $n$ is even and $\lambda$ is odd and (for the sake of contradiction) that $B_n(\lambda)$ has a balanced diagonal $T$. 
For each cell $(r,c)$ in $B_n(\lambda)$ containing entry $e$,   
let 
$$\Delta(r,c)=\lceil r/\lambda \rceil +\lceil c/\lambda \rceil - e.$$
Observe that $\Delta(r,c)\equiv 0$ (mod $n$)  for each cell $(r,c)$ in $B_n(\lambda)$. 
However, since $T$ is a balanced diagonal, each column appears exactly once and each entry appear exactly
$\lambda$ times in $T$, so: 
$$\sum_{(r,c)\in T} \Delta(r,c)= 
\sum_{r\in N(n\lambda)} 
\lceil r/\lambda \rceil 
= 
\lambda \binom{n}{2} \equiv \frac{n}{2} \pmod{n},$$
a contradiction. 
\end{proof}

Finally we show the existence of frequency squares {\em not} equivalent to $L(\lambda)$ for some Latin square $L$, which do not contain balanced diagonals. 
Let $L$ be a frequency square of type $F(m\lambda;\lambda^m)$ and let $\alpha$ be a divisor of $m$.
Let $L_{\alpha}$ be the frequency square of type $F(m\lambda;(\lambda\alpha)^{m/\alpha})$ 
formed by replacing each entry $e$ of $L$ with $\lceil e/\alpha \rceil$. 
Observe the following. 

\begin{lemma}
If there is a balanced diagonal in a frequency square $L$ of type $F(m\lambda;\lambda^m)$ and $\alpha$ divides $m$, there is a balanced diagonal in the corresponding set of cells in
$L_{\alpha}$. 
\end{lemma}
 
This lemma allows us to construct many classes of frequency squares without balanced diagonals (without directly creating any counterexamples to Conjecture \ref{bigandbold}). For example, 
if~$L$ is a frequency square of dimensions $2m\times 2m$, where $m$ is odd, and all odd entries occur in the top-left or bottom right quadrants (with even entries in the other quadrants),
then~$L_{m}$ is equivalent to $A_{2m}$, which does not have a balanced diagonal by Lemma \ref{tobe}; hence $L$ does not have a balanced diagonal.

  \let\oldthebibliography=\thebibliography
  \let\endoldthebibliography=\endthebibliography
  \renewenvironment{thebibliography}[1]{%
    \begin{oldthebibliography}{#1}%
      \setlength{\parskip}{0.4ex plus 0.1ex minus 0.1ex}%
      \setlength{\itemsep}{0.4ex plus 0.1ex minus 0.1ex}%
  }%
  {%
    \end{oldthebibliography}%
  }


\end{document}